\documentclass[oneside,reqno,12pt]{amsart}
\usepackage[greek,english]{babel}
\usepackage{fullpage}
\usepackage{amsthm}
\usepackage{amsbsy}
\usepackage{amsfonts}
\usepackage{graphicx}
\newtheorem{thm}{Theorem}

\newtheorem{pro}{Proposition}
\newtheorem{rem}{Remark}

\numberwithin{equation}{section} \numberwithin{lem}{section}
\numberwithin{thm}{section} \numberwithin{cor}{section}
\numberwithin{pro}{section} \numberwithin{rem}{section}

\def\ep{\varepsilon}

\begin{document}
\title{Phase transition in a Rabi coupled two-component Bose-Einstein condensate}
\author{Amandine Aftalion}
\address{Ecole des Hautes Etudes en Sciences Sociales, PSL Research University, CNRS UMR 8557, Centre d'Analyse et de Math\'ematique Sociales, 54 Boulevard Raspail, 75006 Paris, France.}
\email{amandine.aftalion@ehess.fr}

\author{Christos Sourdis}
\address{Institute of Applied and Computational Mathematics, Foundation of Research and Technology of Hellas, Herakleion,
Crete, Greece.}
\email{sourdis@uoc.gr}

\date{\today}
\begin{abstract}
This paper deals with the study of the phase transition of the wave functions of a segregated two-component
 Bose-Einstein condensate
   under Rabi coupling. This
  yields a system of two coupled ODE's where the Rabi coupling is linear in the other wave function
   and acts against segregation. We prove estimates
  on the asymptotic behaviour of the wave functions, as the strength of the interaction gets strong or weak. We also derive limiting problems in both cases.

\end{abstract}
 \maketitle

%\tableofcontents

\section{Introduction}\label{secIntro}%\input{introaacs.tex}
%%%%%%%%%%%%%%%%%%%%%%%%%%%%%%%%%%%%%%%%%%%%%%%%%%%%%%%%%%%%%%%%%%%%%%%%%%%%%%%%%%%%%%%%%%%%%%%%%%%
\subsection{The problem}

Recently, there has been a huge interest, from the experimental \cite{hal98,matexp}, numerical \cite{am,nitta,dror,sp,qustring,usui}  and mathematical \cite{AL,alamaARMA15,berestycki-wei2012,berestycki2,goldman2015phase} point of view into two component Bose-Einstein condensates. Indeed, the experimental realization of such systems provide opportunities  to explore the rich physics encompassed in it. Two component condensates can interact on the one hand through intercomponent coupling on the modulus, but also through spin orbit coupling. In this paper, we are interested  in a one body coherent Rabi coupling, which provides similar interactions to Josephson coupling in superconductors. This leads to the energy minimization depending on the wave functions $\psi_1$ and $\psi_2$:
\begin{equation}\label{ener2d}
E_\Lambda(\psi_1,\psi_2)=\int_{B}^{}\left\{\sum_{k=1}^{2}\left[\frac{1}{2}|\nabla \psi_k|^2
	+\frac{g_k}{4\ep^2}|\psi_k|^4\right]+\frac{\Lambda}{2\ep^2}|\psi_1|^2|\psi_2|^2
-\frac{\omega}{2\ep^2}(\psi_1\psi_2^*+\psi_1^*\psi_2)\right \},
\end{equation} where $\omega$ denotes the Rabi frequency, $g_k$ the intracomponent coupling, $\Lambda$ the intercomponent coupling, $B$ the unit disc, and $\ep$ is related to the inverse of the number of particules, and therefore is small. We refer to \cite{am,dror} for an introduction to the model and physics references.

The simulations of \cite{am} lead  to phase transitions and vortex sheets that we want to analyze here. We will focus on the 1D phase transition corresponding to the minimization of (\ref{ener2d}) on a 1D interval, close to the interface in the case $g_k=1$. It corresponds to a rescaling in a boundary layer of size $\ep$.

The aim of this paper is therefore to study positive solutions of the system
\begin{equation}\label{eqEq}
\left\{	\begin{array}{rcl}
	 {u}''&=&u^3-u+\Lambda v^2 u-\omega v, \\
&&	\\
	  {v}''&=&v^3-v+\Lambda u^2 v-\omega u,
	\end{array}\right.
	\end{equation}
satisfying
\begin{equation}\label{eqConnect}
(u,v)\to (\bar{u},\bar{v})\ \textrm{as}\ x\to -\infty,\ \   (u,v)\to (\bar{v},\bar{u}) \ \textrm{as}\ x\to +\infty,
\end{equation}
 where $\bar{u},\bar{v}$ are positive numbers to be determined later. The range of values of the positive parameters $\Lambda,\ \omega$ will be discussed in the sequel. This is a heteroclinic connection problem.  The segregation case corresponds to
	  \begin{equation}\label{eqHessian}\Lambda>1\end{equation} and we will study the  limits $\Lambda \to 1$ and
 $\Lambda \to \infty$. Let us point out that in the case $\omega =0$, the solution goes to $(0,1)$ and $(1,0)$ at $\pm \infty$  and this problem has been analyzed in \cite{aftalionSourdis,farina2017monotonicity,sourdisweak}. In \cite{farina2017monotonicity}, it is proved by the moving plane method that the solution is unique, $u'>0$ and $v'<0$. The asymptotic behaviour for large $\Lambda$ has been studied in \cite{aftalionSourdis}: the solution approaches the hyperbolic tangent in the half space while the inner solution is given by a simpler system analyzed in \cite{berestycki-wei2012,berestycki2}. On the other hand, when $\Lambda$ gets to one,  the geometric singular perturbation theory leads to the analysis of the problem on a limiting manifold after a change of function $R=\sqrt{u^2+v^2}$ and $u=R \cos\varphi/2$, $v=R \sin \varphi/2$ and asymptotic results are proved in \cite{sourdisweak}.

 In the case where $\omega$ is not zero, the situation is very different because the limits at infinity are not $(0,1)$ and $(1,0)$ but positive values $(\bar{u},\bar{v})$ and $(\bar{v},\bar{u})$ which are solutions of
 \begin{equation}\label{eql1l2}
	\left\{
	\begin{array}{c}
u(u^2+\Lambda v^2-1)-\omega v=0,\\
v(v^2+\Lambda u^2-1)-\omega u=0.
\end{array}
	\right.
	\end{equation} This comes from the fact that the Rabi coupling mollifies segregation. This  system yields \begin{equation}\label{uv} uv =\frac {\omega}{\Lambda -1},\ u^2+v^2=1.\end{equation} This has a solution if and only if \begin{equation}\label{eqIneq1}
\frac{\omega}{\Lambda-1}<\frac{1}{2},
\end{equation} and the solutions are %\begin{equation}\label{eqEquil2}
 $(\bar{u},\bar{v})$
 %\ \ \textrm{and} \ \ 
  and $(\bar{v},\bar{u})$,
%\end{equation}
where
\begin{equation}\label{equvEq}
\bar{u}=\sqrt{\frac{1-\sqrt{1-\frac{4\omega^2}{(\Lambda-1)^2}}}{2}}\ \ \textrm{and}\ \
\bar{v}=\sqrt{\frac{1+\sqrt{1-\frac{4\omega^2}{(\Lambda-1)^2}}}{2}}.
\end{equation}
 We will provide more details on this in Section \ref{secEquil}.
  When $\Lambda>1$ and (\ref{eqIneq1}) holds,  segregation is not complete as the components coexist in some parts of the domain, but yet there is an interface, which is at the center of this paper.
 We will prove existence of solutions in this regime and study their asymptotic behaviour.

\subsection{Main results} We first focus on the strong segregation case when $\Lambda$ is large. Because of
 (\ref{uv}), this has a non trivial behaviour in $u$ and $v$ if $\omega /(\Lambda -1)$ is away from zero. We therefore assume that $\omega /(\Lambda -1)\sim c_0$ which is not zero and (\ref{eqIneq1}) holds.
 The main result of the paper in the case of large $\Lambda$ is the following:
\begin{thm}\label{thmMain}
Let $\Lambda>1$ be sufficiently large and \begin{equation}\label{eqomega2}
\frac{\omega}{\Lambda-1}=c_0+\tilde{\omega}\left(\frac{1}{\sqrt{\Lambda-1}}\right)
\end{equation}for some fixed \begin{equation}\label{eqc0}c_0\in\left(0,\frac{1}{2}\right),\end{equation}
where $\tilde{\omega}$ is a  $C^2[0,\infty)$  function, independent of $\Lambda$, such that \begin{equation}\label{eqomega00}\tilde{\omega}(0)=0.\end{equation}Then, there exists a solution
                                                        $(u_\Lambda,v_\Lambda)$ of (\ref{eqEq})-(\ref{eqConnect}), where $\bar{u},\bar{v}$ are as in (\ref{equvEq}), such that
                                                        \begin{equation}\label{eqThm1}
                                                          u_\Lambda(-x)\equiv v_\Lambda(x),
                                                        \end{equation}\begin{equation}\label{eqThm2}
                                                                        u_\Lambda'>0,
                                                                      \end{equation}
                                                                      and
\begin{equation}\label{eqThm3}
  u_\Lambda(x)-u_0(x)=O\left(\frac{1}{\sqrt{\Lambda}}\right),
\end{equation}
\begin{equation}\label{eqThm3+}
  \left|u_\Lambda(x)-\bar{u}\right|\cdot\left|u_\Lambda(x)-\bar{v}\right|+\left|u_\Lambda'(x) \right|=O(1)e^{-\left(\sqrt{2}\sqrt{1-4c_0^2}+O\left(\frac{1}{\sqrt{\Lambda}}\right)\right)|x|},
\end{equation}
uniformly in $\mathbb{R}$, as $\Lambda\to \infty$, where
$u_0$ is the unique solution of
\begin{equation}\label{eqLimitPb1}
 u'=  \frac{u^2-u^4- c_0^2 }{\sqrt{2}\sqrt{u^4+c_0^2}},\ \ u(0)=\sqrt{c_0}.
\end{equation}
Furthermore, we have
\begin{equation}\label{eqThm4}
  u_\Lambda(x)v_\Lambda(x)-\frac{\omega}{\Lambda-1}=O\left(\frac{1}{\Lambda}\right)e^{-\left(\sqrt{2}\sqrt{1-4c_0^2}+O\left(\frac{1}{\sqrt{\Lambda}}\right)\right)|x|},
\end{equation}
uniformly in $\mathbb{R}$, as $\Lambda\to \infty$.

\end{thm}
\begin{rem}
  The simplest case where $\omega=c_0\Lambda$ can be put in the form (\ref{eqomega2}) by choosing $\tilde{\omega}(\epsilon)=c_0 \epsilon^2$.
\end{rem}
We note that solutions of (\ref{eqEq}) are such that the  Hamiltonian
	\begin{equation}\label{eqHamil}
  H(u,v)=\frac{(u')^2}{2}+\frac{(v')^2}{2}-\frac{(1-u^2-v^2)^2}{4}-\frac{\Lambda-1}{2}\left(uv-\frac{\omega}{\Lambda-1} \right)^2
\end{equation} is constant. This constant is equal
 to zero along solutions that satisfy (\ref{eqConnect}).
In the limit when $\omega/(\Lambda -1)$ is $c_0$, then we are going to prove that $u_\Lambda v_\Lambda$ is asymptotically equal to $\omega/(\Lambda -1)$ so that the last term in (\ref{eqHamil}) becomes negligible. In order to find the limiting function $u_0$, we can therefore  formally replace $v$ by $c_0/u$ in (\ref{eqHamil}) and find that it is given by (\ref{eqLimitPb1}) which will be detailed in Proposition \ref{proU0} below.

Now we move to the other extreme case of weak segregation when $\Lambda$ tends to 1:
\begin{thm}\label{thm2}
  Let $\Lambda>1$ be sufficiently close to $1$ and
  \begin{equation}\label{eqomegaDef2}
  \frac{\omega}{\Lambda-1}=c_0+\tilde{\omega}\left(\sqrt{\Lambda-1}\right)
  \end{equation}where as before $c_0$ satisfies (\ref{eqc0})
 and $\tilde{\omega}$ is as in Theorem \ref{thmMain}. Then, there exists a solution $(u_\Lambda,v_\Lambda)$ of (\ref{eqEq})-(\ref{eqConnect}), where
$\bar{u},\bar{v}$ are as in (\ref{equvEq}), such that (\ref{eqThm1}) holds and
\[
u_\Lambda^2(x)+v_\Lambda^2(x)=1+O(\Lambda-1)e^{-\left(\sqrt{1-4c_0^2}\sqrt{\Lambda-1}+O(\Lambda-1)\right)|x|},
\]
uniformly in $\mathbb{R}$, as $\Lambda\to 1^+$. Furthermore, the angle \[\frac{\varphi_\Lambda}{2}=\tan^{-1}\left(\frac{v_\Lambda}{u_\Lambda}\right)\] satisfies
\[
\varphi_\Lambda(x)-\varphi_0\left(x\sqrt{\Lambda-1} \right)=O\left(\sqrt{\Lambda-1} \right),
\]uniformly in $\mathbb{R}$, as $\Lambda\to 1^+$, where $\varphi_0$ is the heteroclinic solution of\begin{equation}\label{eqLimitPb2} \varphi_0'=2c_0-\sin\varphi_0,\ \ x\in \mathbb{R},\end{equation}
such that $\varphi_0'<0$, $\varphi_0(0)=\frac{\pi}{2}$ and $\varphi_0(-x)+\varphi_0(x)\equiv \pi$. Moreover, \[\varphi_\Lambda'<0,\ \ x\in \mathbb{R},\]
and \[\left|u_\Lambda(x)-\bar{u}\right|+\left|v_\Lambda(x)-\bar{v}\right|=O\left(e^{\left(\sqrt{1-4c_0^2}\sqrt{\Lambda-1}+O(\Lambda-1)\right)x}\right),\] uniformly in $x<0$, as $\Lambda\to 1^+$.
\end{thm}
The proof relies again on the conservation of Hamiltonian but this time written in polar coordinates $(R,\varphi)$ where $R^2=u^2+v^2$ and $u=R\cos\varphi /2$, $v=R\sin\varphi /2$:
 \[
H(R,\varphi)= \frac{(R')^2}{2}+\frac{R^2}{8}(\varphi')^2-\frac{1}{4}(1-R^2)^2-\frac{\Lambda-1}{2}\left(\frac{R^2}{2}\sin\varphi-\frac{\omega}{\Lambda-1} \right)^2.
\] After rescaling and proving, at the limit when $\omega/(\Lambda -1)$ tends to $c_0$, that $R$ tends to 1,  and the last term in $H$ vanishes, we find (\ref{eqLimitPb2}).

\subsection{Method of proof} Our approach for showing Theorems \ref{thmMain} and \ref{thm2} is the same. In each case  we define a suitable small parameter $\varepsilon=\varepsilon(\Lambda)$ and make a convenient change of coordinates in order to write (\ref{eqEq}) as a   singularly perturbed system in slow-fast form with two slow and two fast variables. Loosely speaking, for Theorem \ref{thmMain} we   use a change of coordinates that straightens the hyperbola $\{uv=\frac{\omega}{\Lambda-1}\}$, while for Theorem \ref{thm2} we  employ polar coordinates. Then, in the resulting slow-fast formulation   we can apply standard theorems of geometric singular perturbation theory (see \cite{kuhen} and the references therein). We find that in both cases the dynamics can be reduced on  two-dimensional invariant manifolds $\mathcal{M}_\varepsilon$, that vary smoothly for small $\varepsilon\geq 0$, with  the flow on them being determined by   smooth regular perturbations of the limit problems (\ref{eqLimitPb1}) and (\ref{eqLimitPb2}), respectively. On the limiting manifold $\mathcal{M}_0$, there exists a singular heteroclinic connection between the equilibria corresponding to the $\varepsilon=0$ limits of $(\bar{u},\bar{v})$ and $(\bar{v},\bar{u})$, which are saddles with two-dimensional stable and unstable manifolds. Clearly the intersection of the latter manifolds cannot be transverse in the ambient space $\mathbb{R}^4$ nor on $\mathcal{M}_0$. We establish the persistence of the singular heteroclinic connection on $\mathcal{M}_\varepsilon$ for small $\varepsilon$ by exploiting the conservation of the Hamiltonian (\ref{eqHamil}).

More precise estimates will be detailed in Theorems \ref{thm1} and \ref{thm11} respectively.

\subsection{Open questions}A natural question to ask is whether the solution we have found is unique. It would be very nice to have such a proof using moving plane methods, in particular extending \cite{farina2017monotonicity}. This would require precise bounds on $u^2+v^2-1$ on the one hand and $uv-\omega/(\Lambda -1)$ on the other hand. In the case $\omega=0$, we have a uniqueness proof in \cite{aftalionSourdis} which is based on the continuation method starting from $\Lambda=3$ and using the nondegeneracy of the linearized operator. But this cannot be applied here, since our result only holds for large $\Lambda$ or $\Lambda$ close to 1.

The other limit,  when $\omega/(\Lambda -1)$ is small and $\Lambda$ is large is not treated in this paper. It should display  segregation but in a more regular manner than for $\omega=0$, since
 at leading order we expect $u v\sim \omega/(\Lambda-1) $.

 We have studied the case  when $\omega/(\Lambda -1)$ is less than 1/2. On the other hand,
 when $\omega/(\Lambda -1)$ is bigger than 1/2, we expect coexistence of the components that is the ground state will be given by $u=v$: the Rabi coupling should overcome segregation.

\subsection{Outline of the paper} In Section \ref{secEquil} we will find the equilibria of (\ref{eqEq}). The proof of Theorem \ref{thmMain} will be carried out in Section \ref{secSeparation}, while that of Theorem \ref{thm2} will be carried out in Section \ref{secWeakSepar}.

\section{Equilibria}\label{secEquil}
To find the equilibria of (\ref{eqEq}), we need to solve the following algebraic system:
\begin{equation}\label{eqAlgeb}
		\begin{array}{lcr}
	 u^3-u+\Lambda v^2 u-\omega v&=&0, \\
	  v^3-v+\Lambda u^2 v-\omega u&=&0.
	\end{array}
	\end{equation}%(see also \cite{dror}).

Multiplying the first equation by $u$, the second by $v$, then subtracting and adding the resulting equations leads to the system
\begin{equation}\label{eqlatter}
		\begin{array}{lcr}
(u^2-v^2)(u^2+v^2-1)&=&0,\\ &&	\\
  \left(u^2+v^2-\frac{1}{2}\right)^2 +2 (\Lambda -1) u^2 v^2 -2 \omega uv&=&\frac{1}{4}.\end{array}
\end{equation}
As we are interested in positive solutions of (\ref{eqEq}), we get from the former relation that
\begin{equation}\label{eqAlg1}
u=v\ \ \textrm{or}\ \ u^2+v^2=1.
\end{equation}
The first case in (\ref{eqAlg1}) yields, through (\ref{eqlatter}), the equilibria
\begin{equation}\label{eqEquil1}
(0,0)\ \ \textrm{and}\ \ \left(\sqrt{\frac{1+\omega}{1+\Lambda}},\sqrt{\frac{1+\omega}{1+\Lambda}} \right).
\end{equation}
The second case in (\ref{eqAlg1}) yields either $uv=0$ or $(\bar u, \bar v)$ or $(\bar v,\bar u)$ with (\ref{equvEq}),
 in the case (\ref{eqIneq1}).
We note that if $u$ or $v$ is zero, then because of (\ref{eqlatter}), we have $(1,0)$ and $(0,1)$ as equilibria but these clearly do not satisfy (\ref{eqAlgeb}).
 
 \section{The strong separation limit}\label{secSeparation}
We consider the regime where $\Lambda \to +\infty$ in Theorem \ref{thmMain}.
We  expect  that the product $uv$ of solutions to (\ref{eqEq}), (\ref{eqConnect}) should converge to $c_0$, as $\Lambda \to +\infty$, at least in some weak sense. Therefore, it is natural to define a  new independent variable
\begin{equation}\label{eqhx}
h=uv-\frac{\omega}{\Lambda-1}.
\end{equation}
Then, it follows readily that system (\ref{eqEq}) is equivalent to
\begin{equation}\label{eqhu}\begin{array}{rcl}
    h'' & = & \left[\left(1+\frac{2}{\Lambda-1}\right)(\Lambda-1)h+\frac{2\omega}{\Lambda-1}  \right]\left[u^2+\frac{1}{u^2}\left(h+\frac{\omega}{\Lambda-1} \right)^2 \right] \\
     &  & -2h-\frac{2\omega}{\Lambda-1}+2\frac{u'}{u^2}\left(-\frac{\omega}{\Lambda-1}u'+h'u-hu' \right) \\
      &   &   \\
    u'' & = & u^3-u+\frac{1}{u}\left(h+\frac{\omega}{\Lambda-1} \right)\left(\Lambda h+\frac{\omega}{\Lambda-1} \right) .\end{array}
\end{equation}
Moreover,  conditions (\ref{eqConnect}) become
\begin{equation}\label{eqConnectNew}\begin{array}{c}
                                      (h,u)\to \left(0,\bar{u}\right)\ \textrm{as}\ x\to -\infty, \\ \\
                                      (h,u)\to \left(0,\bar{v}\right) \ \textrm{as}\ x\to +\infty.
                                    \end{array}
\end{equation}
We find that the Hamiltonian
\begin{equation}\label{eqHamilhu}
\begin{split}
  \tilde{H}=&\frac{(u')^2}{2}+\frac{\left(h'u-\left(h+\frac{\omega}{\Lambda-1}\right)u'\right)^2}{2u^4}\\&-\frac{\left(1-u^2-\frac{1}{u^2}\left(h+\frac{\omega}{\Lambda-1}
 \right)^2\right)^2}{4}-\frac{\Lambda-1}{2}h^2
\end{split}\end{equation} derived from (\ref{eqHamil}),
is conserved along solutions of (\ref{eqhu}). In particular,  $\tilde{H}=0$ along solutions that satisfy
(\ref{eqConnectNew}).

\subsection{Slow-fast formulation}
We set
\begin{equation}\label{eq32}
\varepsilon=\frac{1}{\sqrt{\Lambda-1}},
\end{equation}
\begin{equation}\label{eq33}
h=\varepsilon^2 p,\ \ q=\varepsilon p',\ \ z=u'.
\end{equation}
Then, we can write system (\ref{eqhu}) in the following slow-fast form:
\begin{equation}\label{eqSF}\begin{array}{rcl}
\varepsilon p'&=&q,\\
&&\\
    \varepsilon q' & = & \left[\left(1+ 2\varepsilon^2 \right)p+ 2 \left(c_0+\tilde{\omega}(\varepsilon)\right)  \right]\left[u^2+\frac{1}{u^2}\left(\varepsilon^2p+ c_0+\tilde{\omega}(\varepsilon)\right)^2 \right] \\
     &  & -2\varepsilon^2p-2\left(c_0+\tilde{\omega}(\varepsilon)\right)\left(1+\frac{z^2}{u^2}\right) +2\varepsilon \frac{z}{u}q-2\varepsilon^2 \frac{z^2}{u^2}p, \\
      &   &   \\
      u'&=&z,\\
      & &\\
    z' & = & u^3-u+\frac{1}{u}\left(\varepsilon^2p+c_0+\tilde{\omega}(\varepsilon) \right)
    \left[(1+\varepsilon^2)p +  c_0+\tilde{\omega}(\varepsilon) \right]. \end{array}
\end{equation}
This is called the \emph{slow system}.

Moreover, the conditions (\ref{eqConnectNew}) become
\begin{equation}\label{eqConnectNewest}\begin{array}{c}
                                      (p,q,u,z)\to \left(0,0,\bar{u},0\right)\ \textrm{as}\ x\to -\infty, \\ \\
                                      (p,q,u,z)\to \left(0,0,\bar{v},0\right) \ \textrm{as}\ x\to +\infty.
                                    \end{array}
\end{equation}

The eigenvalues of the linearization of (\ref{eqSF}) at the equilibria $\left(0,0,\bar{u},0\right)$ and $\left(0,0,\bar{v},0\right)$ are real and  given by
\begin{equation}\label{eqEV-}
\pm\frac{1}{\varepsilon}+O(1), \ \ \pm  \sqrt{2}\sqrt{1-4c_0^2}+O(\varepsilon)\ \ \textrm{as}\ \ \varepsilon \to 0.
\end{equation} Therefore, each of these equilibria is a saddle with two-dimensional (global) stable and unstable manifolds. In the light of (\ref{eqConnectNewest}), we will be interested in the unstable manifold $W^u(0,0,\bar{u},0)$ of $(0,0,\bar{u},0)$ and the stable manifold $W^s(0,0,\bar{v},0)$ of $(0,0,\bar{v},0)$.

By virtue of (\ref{eqHamil}), we find that the Hamiltonian
\begin{equation}\label{eqHamilpq}
\begin{split}
  \hat{H}=&\frac{z^2}{2}+\frac{\left(\varepsilon q u-\left(\varepsilon^2p+c_0+\tilde{\omega}(\varepsilon)\right)z\right)^2}{2u^4}\\&-\frac{\left(1-u^2-\frac{1}{u^2}\left(\varepsilon^2p+c_0+\tilde{\omega}(\varepsilon)
 \right)^2\right)^2}{4}-\frac{\varepsilon^2}{2}p^2
\end{split}\end{equation}
is conserved along solutions of (\ref{eqSF}). In particular,  $\hat{H}=0$ holds along solutions that satisfy
one of the asymptotic behaviours in (\ref{eqConnectNewest}). In other words, the following holds:
\begin{equation}\label{eqH=0}
\hat{H}=0\ \ \textrm{on}\ \ W^u(0,0,\bar{u},0)\cup W^s(0,0,\bar{v},0).
\end{equation}

\subsection{The slow (critical) manifold $\mathcal{M}_0$ and the reduced system}\label{secSlowRed}
Formally setting $\varepsilon=0$ in (\ref{eqSF}), and keeping in mind (\ref{eqomega00}), gives us the  \emph{slow limit system}:
\begin{equation}\label{eqSFlimit}\begin{array}{rcl}
0&=&q,\\
&&\\
    0 & = & (p+ 2 c_0  )\left(u^2+\frac{c_0^2}{u^2}\right)-2c_0-2c_0\frac{z^2}{u^2}, \\
      &   &   \\
      u'&=&z,\\
      & &\\
    z' & = & u^3-u+\frac{c_0}{u}
    (p +  c_0). \end{array}
\end{equation}

By solving the  first two equations for $p,q$ we can determine the \emph{slow manifold}:
\begin{equation}\label{eqslowManif}
\mathcal{M}_0=\left\{p=2c_0\frac{u^2+z^2}{u^4+c_0^2}-2c_0,\  \ q=0,\ \ (u,z)\in \mathbb{R}^2,\ u\neq 0\right\}.
\end{equation}

Plugging this in the last two equations, gives us the \emph{reduced system}:
\begin{equation}\label{eqReduc}\begin{array}{rcl}
      u'&=&z,\\
      & &\\
    z' & = & u^3-u+\frac{c_0^2}{u}\left(2\frac{u^2+z^2}{u^4+c_0^2}-1 \right). \end{array}
\end{equation}

The above system has the equilibria $(\bar{u}_0,0)$ and $(\bar{v}_0,0)$, where
\begin{equation}\label{equ0v0}
\bar{u}_0=\sqrt{\frac{1-\sqrt{1-4c_0^2}}{2}}\ \ \textrm{and}\ \
\bar{v}_0=\sqrt{\frac{1+\sqrt{1-4c_0^2}}{2}}.
\end{equation}We note that these are the limits of $\bar{u}$ and $\bar{v}$ from (\ref{equvEq}), respectively.
The eigenvalues of the corresponding linearizations  at both of these equilibria are
$\pm  \sqrt{2}\sqrt{1-4c_0^2}$ (keep in mind also (\ref{eqEV-})).
Therefore, each of these equilibria is a saddle with one-dimensional (global) stable and unstable manifolds.
We will be interested in the unstable manifold $W^u(\bar{u}_0,0)$ of $(\bar{u}_0,0)$ and the stable manifold $W^s(\bar{v}_0,0)$ of
$(\bar{v}_0,0)$. The former manifold is tangent to $(1,  \sqrt{2}\sqrt{1-4c_0^2})$ at $(\bar{u}_0,0)$, while the latter is tangent to $(1,-  \sqrt{2}\sqrt{1-4c_0^2})$ at $(\bar{v}_0,0)$.
 \subsection{The singular heteroclinic connection}\label{subsecUo}
By setting  $\omega/(\Lambda-1)$ equal to $c_0$ and $v$ equal to $c_0/u$ in (\ref{eqHamil}) (since $uv=c_0$ on $\mathcal{M}_0$), we can write the limiting hamiltonian
\begin{equation}\label{eqHamiltReduc}
H_0=\left(1+\frac{c_0^2}{u^4} \right)\frac{z^2}{2}-\frac{1}{4}\left(1-u^2-\frac{c_0^2}{u^2}\right)^2.
\end{equation}
It is indeed conserved along solutions of (\ref{eqReduc}) and equal to 0 if $u$ connects the equilibria in (\ref{equ0v0}).

Solutions of (\ref{eqReduc}) with $H_0=0$ satisfy one of the following first order ODEs:\begin{equation}\label{eqRedInt}
u'=\pm \frac{u^2-u^4- c_0^2 }{\sqrt{2}\sqrt{u^4+c_0^2}}.
\end{equation}
 In fact, we have the following simple proposition.

\begin{pro}\label{proU0}
  The unique solution $u_0$ of (\ref{eqRedInt}) with the plus sign and
\begin{equation}\label{eqHalf}
  u_0(0)=\sqrt{c_0},
\end{equation}satisfies \begin{equation}\label{eqRedIntHet}
u_0\to \bar{u}_0\ \textrm{as}\ x\to -\infty,\ \ u_0\to \bar{v}_0\ \textrm{as}\ x\to +\infty,
\end{equation} \begin{equation}\label{equ0'}
               z_0=u_0'>0,
             \end{equation} and
\begin{equation}\label{eqSymm}
  u_0(x)u_0(-x)\equiv c_0.
\end{equation}Moreover, $u_0$ is the unique modulo translations solution of (\ref{eqRedInt}) with the plus sign and (\ref{eqRedIntHet}).\end{pro}\begin{proof}
Since  equation (\ref{eqRedInt}) with the plus sign is a first order ODE, solutions with initial value between the consecutive equilibria $\bar{u}_0<\bar{v}_0$ are increasing and satisfy
(\ref{eqRedIntHet}). The uniqueness properties follow directly from the uniqueness of the initial value problem for (\ref{eqRedInt}).
                         It remains  to show (\ref{eqSymm}).
This  follows by observing  that if $u$ satisfies (\ref{eqRedInt}) with the plus sign, then
\[
\tilde{u}(x)=\frac{c_0}{u(-x)}
\]
also satisfies the same equation. Thus, since $u_0(0)=\tilde{u}_0(0)$, we obtain the desired relation.
                       \end{proof}

  For future reference, we note that differentiation of (\ref{eqSymm}) yields
             \begin{equation}\label{eqSymm2}
               z_0(-x)\equiv c_0\frac{z_0(x)}{u_0^2(x)}.
             \end{equation}

In regards with (\ref{eqReduc}), the trajectory $(u_0,z_0)$ lies in the intersection of the unstable manifold $W^u(\bar{u}_0,0)$ of $(\bar{u}_0,0)$ and the stable manifold $W^s(\bar{v}_0,0)$ of $(\bar{v}_0,0)$. We will only be concerned with these parts of the aforementioned invariant manifolds.
The lifting of $(u_0,z_0)$ on $\mathcal{M}_0$ is called a \emph{singular heteroclinic connection}.% and we denote it by $(p_0,q_0,u_0,z_0)$.

%The eigenvalues of the linearization of (\ref{eqReduc}) at both of the equilibria $(\bar{u}_0,0)$ and $(\bar{v}_0,0)$ are
%$\pm   \sqrt{2}\sqrt{1-4c_0^2}$. Therefore, each of these equilibria has   one-dimensional stable and unstable manifolds.

\subsection{Normal hyperbolicity of the slow manifold}\label{subsecNH} The slow manifold $\mathcal{M}_0$ is \emph{normally hyperbolic}
if and only if the linearization of the righthand side of the first two equations in (\ref{eqSFlimit}) with respect to $(p,q)$, at any point $(p,q,u,z)$ on $\mathcal{M}_0$, does not have eigenvalues on the imaginary axis. The aforementioned linearization is
\[
\left(\begin{array}{cc}
        0 & 1 \\
        u^2+\frac{c_0^2}{u^2} & 0
      \end{array}
 \right),
\]whose eigenvalues are $\lambda_\pm=\pm\sqrt{u^2+\frac{c_0^2}{u^2}}$. Hence, the slow manifold $\mathcal{M}_0$ is normally hyperbolic.

\subsection{Local persistence of $\mathcal{M}_0$: The invariant manifold $\mathcal{M}_\varepsilon$}\label{secMeps} Let $\mathcal{K}\subset \left\{u>0,\ z\in\mathbb{R}\right\}$ be a compact, simply connected domain which contains the heteroclinic orbit $(u_0,z_0)$, and whose boundary is a
$C^\infty$ curve. As a  consequence of Fenichel's first theorem (see \cite{fenichelJDE}, \cite{jones1995geometric} or \cite[Ch. 3]{kuhen}), we deduce that
the restriction of $\mathcal{M}_0$ over $\mathcal{K}$ perturbs smoothly for small $\varepsilon \geq 0$ to a locally invariant, normally hyperbolic manifold $\mathcal{M}_\varepsilon$ for (\ref{eqSF}).
More precisely,  given an integer $m\geq 1$,  there is an $\varepsilon_0>0$ and functions $h_i(u,z,\varepsilon)\in C^m\left(\mathcal{K}\times [0,\varepsilon_0) \right)$,
$i=1,2$,  such that the manifold $\mathcal{M}_\varepsilon$ described by
\begin{equation}\label{eqh}
p=2c_0\frac{u^2+z^2}{u^4+c_0^2}-2c_0+\varepsilon h_1(u,z,\varepsilon),
\ \ q=\varepsilon h_2(u,z,\varepsilon),\ \ (u,z)\in\mathcal{K},\end{equation}
is a normally hyperbolic, locally invariant manifold for (\ref{eqSF}) if $\varepsilon\in (0,\varepsilon_0)$.

 By the normal hyperbolicity of $M_\varepsilon$, and by possibly decreasing the value of $\varepsilon_0>0$, we deduce that the equilibria $(0,0,\bar{u},0)$ and $(0,0,\bar{v},0)$ of (\ref{eqSF})  lie on $\mathcal{M}_\varepsilon$, i.e.,
\begin{equation}\label{eqEquiv0}h_i\left(\bar{u},0,\varepsilon\right)=0,\ \ h_i\left(\bar{v},0,\varepsilon\right)=0,\ \ i=1,2,\ \ \varepsilon \in [0,\varepsilon_0).\end{equation} %This is because every invariant set of (\ref{eqSF}) in a sufficiently small $\varepsilon$-independent neighborhood of $\mathcal{M}_0$ must be on $\mathcal{M}_\varepsilon$.

\subsection{Theorem \ref{thm1}}
The above leads us to the main result of this section.
\begin{thm}\label{thm1}
For each $\varepsilon>0$ sufficiently small, there is a heteroclinic solution $(p_\varepsilon,q_\varepsilon,u_\varepsilon,z_\varepsilon)$ of (\ref{eqSF}) satisfying (\ref{eqConnectNewest}) which lies on $\mathcal{M}_\varepsilon$ such that
\begin{equation}\label{eqEstims1}
     u_\varepsilon=u_0+O(\varepsilon),\ \  \   z_\varepsilon=z_0+O(\varepsilon),
   \end{equation}and
   \begin{equation}\label{eqproduct}
     |u_\varepsilon-\bar{u}|\cdot|u_\varepsilon-\bar{v}|+|z_\varepsilon|=O(1)e^{-\left(\sqrt{2}\sqrt{1-4c_0^2}+O(\varepsilon)\right)|x|},
   \end{equation}
uniformly in $\mathbb{R}$, as $\varepsilon\to 0$, where $(u_0,z_0)$ is as in Proposition \ref{proU0}. Furthermore,
\begin{equation}\label{eqz>0}
  z_\varepsilon>0.
\end{equation}
More precisely, the following estimates hold:
\begin{equation}\label{eqexp22}
   \begin{array}{c}
     p_\varepsilon=2c_0\frac{u_\varepsilon^2+z_\varepsilon^2}{u_\varepsilon^4+c_0^2}-2c_0+O(\varepsilon)e^{-\left(\sqrt{2}\sqrt{1-4c_0^2}+O(\varepsilon)\right)|x|}, \\ \\

     q_\varepsilon=O(\varepsilon)e^{-\left(\sqrt{2}\sqrt{1-4c_0^2}+O(\varepsilon)\right)|x|},
   \end{array}
\end{equation}
uniformly in $\mathbb{R}$, as $\varepsilon\to 0$.
\end{thm}
\begin{proof}
 Substituting (\ref{eqh}) in the last two equations of (\ref{eqSF}) determines the flow of the restriction of the latter system on its invariant manifold $\mathcal{M}_\varepsilon$. The resulting system is a smooth $O(\varepsilon)$-perturbation of the reduced system (\ref{eqReduc}). For definiteness, we will refer to it as the $\varepsilon$-reduced system.

 The unstable manifold $W^u(\bar{u}_0,0)$ of $(\bar{u}_0,0)$ and the stable manifold $W^s(\bar{v}_0,0)$ of $(\bar{v}_0,0)$ for (\ref{eqReduc}) perturb smoothly to
 the unstable manifold $W^u(\bar{u},0)$ of $(\bar{u},0)$ and the stable manifold $W^s(\bar{v},0)$ of $(\bar{v},0)$ for the $\varepsilon$-reduced problem, respectively.
Our goal is to show that $W^u(\bar{u},0)$ and  $W^s(\bar{v},0)$ meet for sufficiently small $\varepsilon>0$. Thus, they have to coincide since they are
one-dimensional.  The desired heteroclinic connection for (\ref{eqSF}) will be provided by their lifting on $\mathcal{M}_\varepsilon$.
%Indeed, clearly such a heteroclinic connection would satisfy the estimates in (\ref{eqEstims1}) minus the exponential decay factor in the remainders. In turn,  this can be included by using the equations and recalling the exponential decay rates for the reduced problem from the end of  Subsection \ref{secSlowRed}.

Let us show that $W^u(\bar{u},0)$ and  $W^s(\bar{v},0)$ meet on the line $l=\{u=u_0(0),\ z\in \mathbb{R}\}$ (recall (\ref{eqHalf})). We note that there is nothing special about the choice of this line, the important thing is that it is  transverse to $(u_0,z_0)$ (recall (\ref{equ0'})).
As we have said, the manifolds $W^u(\bar{u},0)$ and  $W^s(\bar{v},0)$ depend smoothly on $\varepsilon\geq 0$ small. Thus, they intersect the line $l$ at some points $\left(u_0(0),z_\varepsilon^-\right)$ and $\left(u_0(0),z_\varepsilon^+\right)$, respectively, such that
\begin{equation}\label{eqZdiff}
  z_\varepsilon^\pm-z_0(0)=O(\varepsilon)\ \ \textrm{as}\ \varepsilon\to 0,
\end{equation}
(recall (\ref{eqHalf})). Keep in mind that our goal is to show that \begin{equation}\label{eqZ=}
                                                                      z_\varepsilon^-=z_\varepsilon^+.
                                                                    \end{equation} To this end, let $\left(p_\varepsilon^-,q_\varepsilon^-,u_0(0),z_\varepsilon^-\right)$ and $\left(p_\varepsilon^+,q_\varepsilon^+,u_0(0),z_\varepsilon^+\right)$ be the liftings on $\mathcal{M}_\varepsilon$ of
$\left(u_0(0),z_\varepsilon^-\right)$ and $\left(u_0(0),z_\varepsilon^+\right)$  via (\ref{eqh}), respectively.
Thanks to the aforementioned smoothness with respect to small $\varepsilon\geq 0$, it  holds
\begin{equation}\label{eqpqDiff}
  p_\varepsilon^\pm-p_0=O(\varepsilon)\ \textrm{and}\ q_\varepsilon^\pm-q_0=O(\varepsilon)\ \textrm{as}\ \varepsilon\to 0,
\end{equation}
where $(p_0,q_0)$ is the image of $\left(u_0(0),z_0(0) \right)$ on the graph of $\mathcal{M}_0$.
Since $\left(u_0(0),z_\varepsilon^-\right)\in W^u(\bar{u},0)$, we infer from (\ref{eqh})-(\ref{eqEquiv0}) that
\begin{equation}\label{eqW-}
  \left(p_\varepsilon^-,q_\varepsilon^-,u_0(0),z_\varepsilon^-\right)\in W^u(0,0,\bar{u},0)\cap \mathcal{M}_\varepsilon.
\end{equation}
Similarly we have
\begin{equation}\label{eqW+}
  \left(p_\varepsilon^+,q_\varepsilon^+,u_0(0),z_\varepsilon^+\right)\in W^s(0,0,\bar{v},0)\cap \mathcal{M}_\varepsilon.
\end{equation}
Hence, in view of (\ref{eqH=0}), we find that both  $\left(p_\varepsilon^-,q_\varepsilon^-,u_0(0),z_\varepsilon^-\right)$
and $\left(p_\varepsilon^+,q_\varepsilon^+,u_0(0),z_\varepsilon^+\right)$ satisfy  the equation \begin{equation}\label{eqH=0(new)}
                                                                                      \hat{H}(p,q,u,z)=0.
                                                                                    \end{equation}

We will show that in some (fixed) neighborhood of $\left(p_0,q_0,u_0(0),z_0(0) \right)$ the algebraic system comprised of the two equations in (\ref{eqh}) and (\ref{eqH=0(new)}) admits a unique solution $(p,q,u,z)$, provided that $\varepsilon>0$ is sufficiently small. Then, taking into account (\ref{eqZdiff}) and (\ref{eqpqDiff}), this would imply the desired relation (\ref{eqZ=}). We will accomplish this by means of the implicit function theorem, applied to the mapping $F:\mathbb{R}^2\times \mathcal{K}\times [0,\varepsilon_0)\to \mathbb{R}^3$ defined by
\begin{equation}\label{eqMapping}
F\left(\begin{array}{c}
         p \\
         q \\
         u \\
         z \\
         \varepsilon
       \end{array}
 \right)=\left(\begin{array}{c}
p- 2c_0\frac{u^2+z^2}{u^4+c_0^2}+2c_0-\varepsilon h_1(u,z,\varepsilon) \\
           q-\varepsilon h_2(u,z,\varepsilon) \\
           \hat{H}(p,q,u,z,\varepsilon)
         \end{array}\right),
\end{equation}
where the set $\mathcal{K}$ was defined in the beginning of Subsection \ref{secMeps}, $h_1,\ h_2$ are as in (\ref{eqh}),  and $\hat{H}$ is as in (\ref{eqHamilpq}).
In view of the comments leading to (\ref{eqh}), and (\ref{eqHamilpq}) (keeping in mind our smoothness assumption on $\tilde{\omega}$), the above mapping is $C^2$ in its domain  of definition, having decreased the value of $\varepsilon_0>0$ if needed.
Keeping in mind (\ref{eqW-}), (\ref{eqW+}) and (\ref{eqH=0(new)}), we find that
\begin{equation}\label{eqF=0}
F\left(p_\varepsilon^\pm,q_\varepsilon^\pm,u_0(0),z_\varepsilon^\pm,\varepsilon\right)=0\ \ \textrm{for small}\ \varepsilon>0.
\end{equation}In fact, the above relation continues to hold for $\varepsilon=0$ as \[F\left(p_0,q_0,u_0(0),z_0(0),0\right)=0. \]
Moreover,
\[
\partial_{pqz}F\left(\begin{array}{c}
                       p \\
                       q \\
                       u \\
                       z \\
                       0
                     \end{array}
 \right)=\left(\begin{array}{ccc}
                 1 & 0 & -4c_0\frac{z}{u^4+c_0^2} \\
                 0 & 1 & 0 \\
                 0 & 0 & z+c_0^2\frac{z}{u^4}
               \end{array}
  \right).
\]
In particular, the above matrix is invertible at $(p_0,q_0,u_0(0),z_0(0),0)$ (recall (\ref{equ0'})).
Hence, we deduce by the implicit function theorem that
there exists a $\delta > 0$ such that, for $\left|u-u_0(0)\right|<\delta$
and $\varepsilon\in [0, \delta)$, the
equation
$F (p, q, u, z, \varepsilon) = (0, 0, 0)$
has at most one solution $(p, q, z)$ such that $|p - p_0| < \delta$, $|q - q_0| < \delta$
and $|z - z_0(0)| < \delta$. Then, applying this property with $u=u_0(0)$, we
infer from (\ref{eqF=0}), having in mind (\ref{eqZdiff}) and (\ref{eqpqDiff}), that the desired relation (\ref{eqZ=}) is true if
$\varepsilon>0$ is sufficiently small.

So far we have shown that there exists a heteroclinic connection for (\ref{eqSF}) on $\mathcal{M}_\varepsilon$ satisfying (\ref{eqEstims1}), after a suitable translation. The exponential decay estimate in (\ref{eqproduct}) follows from local analysis at the equilibria $(\bar{u},\bar{v})$ and $(\bar{v},\bar{u})$ of the $\varepsilon$-reduced problem. Indeed, the linearization of the $\varepsilon$-reduced problem at both equilibria has eigenvalues $\pm \sqrt{2}\sqrt{1-4c_0^2}+O(\varepsilon)$ (recall the last part of Subsection \ref{secSlowRed}).  The estimates in (\ref{eqexp22}) then follow  by recalling (\ref{eqh}) and (\ref{eqEquiv0}).

Lastly, the property (\ref{eqz>0}) is a direct consequence of (\ref{equ0'}) and the fact that $W^u(\bar{u}_0,0)$ and $W^s(\bar{v}_0,0)$ cross $z=0$ transversely at $(\bar{u}_0,0)$ and $(\bar{v}_0,0)$, respectively (recall again the last part of Subsection \ref{secSlowRed}).
\end{proof}

We can also show the local uniqueness of the heteroclinic connection of Theorem \ref{thm1}.
\begin{pro}\label{proUniq}
There exists  a small  fixed neighborhood of  the orbit \[\Gamma_\varepsilon=\left\{\left(p_\varepsilon(x),q_\varepsilon(x),u_\varepsilon(x),z_\varepsilon(x)\right),\ \ x\in\mathbb{R}\right\}\] inside which there is no other  connecting orbit for (\ref{eqSF})-(\ref{eqConnectNewest}) if $\varepsilon>0$ is sufficiently small.
\end{pro}
\begin{proof}
Let us suppose that in some fixed neighborhood of $\Gamma_\varepsilon$ there was another connecting orbit $\tilde{\Gamma}_\varepsilon$ for small $\varepsilon>0$. Then, provided that the aforementioned neighborhood is sufficiently small, the curve $\tilde{\Gamma}_\varepsilon$ would also lie on $\mathcal{M}_\varepsilon$ if $\varepsilon>0$ is sufficiently small (by the same reasoning as for reaching (\ref{eqEquiv0})).
Hence, the    projection of $\tilde{\Gamma}_\varepsilon$ on the $uz$ plane would also be a connecting orbit for the same $\varepsilon$-reduced problem as the corresponding projection of $\Gamma_\varepsilon$. In other words, the aforementioned projections coincide with the one-dimensional intersection $W^u(\bar{u},0)\cap W^s(\bar{v},0)$.
This clearly implies the desired local uniqueness property.
 \end{proof}
\subsection{Proof of Theorem \ref{thmMain}}\label{subsecProof1}\begin{proof}
The desired solution $(u_\Lambda,v_\Lambda)$ is  provided by Theorem \ref{thm1},  keeping track of the definitions (\ref{eqhx}), (\ref{eq32}), (\ref{eq33}) (where with some abuse of notation we identify $u_\Lambda$, $\Lambda\gg 1$ with $u_\varepsilon$, $\varepsilon\ll 1$), and translating it so that
\begin{equation}\label{eqEqual}
  u_\Lambda(0)=v_\Lambda(0).
\end{equation}
We point out that such a translation does not affect the estimates of the aforementioned theorem, which imply the validity of (\ref{eqThm2}), (\ref{eqThm3}), (\ref{eqThm3+}) and (\ref{eqThm4}).

It remains to verify (\ref{eqThm1}).
To this end, the main observation is that the pair \[\left(\tilde{u}_\Lambda(x),\tilde{v}_\Lambda(x)\right)=\left(v_\Lambda(-x),u_\Lambda(-x)\right)\] is also a solution to (\ref{eqEq})-(\ref{eqConnect}). Let $(\tilde{p},\tilde{q},\tilde{u},\tilde{z})$ be the corresponding solution of (\ref{eqSF})-(\ref{eqConnectNewest}) that is given through (\ref{eqhx}), (\ref{eq32}) and (\ref{eq33}) with $(\tilde{u},\tilde{v})$ in place of $(u,v)$.
By virtue of the estimates in Theorem \ref{thm1}, and (\ref{eqSymm}), we find that
\[
\tilde{u}_\varepsilon(x)=v_\varepsilon(-x)=\frac{c_0}{u_\varepsilon(-x)}+O(\varepsilon)=\frac{c_0}{u_0(-x)}+O(\varepsilon)=u_0(x)+O(\varepsilon),
\]
uniformly for $x\in \mathbb{R}$, as $\varepsilon \to 0$.
Similarly
\[
\tilde{z}_\varepsilon=z_0+O(\varepsilon),
\]
uniformly in $\mathbb{R}$, as $\varepsilon \to 0$.
In turn, by virtue of (\ref{eqhx}), (\ref{eq33}), (\ref{eqSymm}), (\ref{eqSymm2}) and (\ref{eqexp22}), we obtain  that
\[
\tilde{p}(x)=p(-x)=p(x)+O(\varepsilon)\ \ \textrm{and}\ \ \tilde{q}(x)=-q(-x)=O(\varepsilon),
\]
uniformly in $\mathbb{R}$, as $\varepsilon \to 0$.
Consequently,  we infer from Proposition \ref{proUniq} and (\ref{eqEqual}) that
\[
(\tilde{p},\tilde{q},\tilde{u},\tilde{z})\equiv (p,q,u,z),
\]
which clearly implies the validity of (\ref{eqThm1}).\end{proof}
 \section{The weak separation limit}\label{secWeakSepar}We consider the regime $\Lambda \to 1^+$ in Theorem \ref{thm2}.\subsection{Slow-fast formulation}
We set as small parameter
\begin{equation}\label{eqEpsilon}
\varepsilon=\sqrt{\Lambda-1},
\end{equation}
and consider the slow variable
\begin{equation}\label{eqxz}
y=\varepsilon x.
\end{equation}
Then, system (\ref{eqEq}) is equivalent to
\begin{equation}\label{eqEqGenEpsilon}
	\left\{
	\begin{array}{rcl}
	 \varepsilon^2 u''&=&u^3-u+  v^2 u+ \varepsilon^2 v^2 u-\omega v, \\
&&	\\
	  \varepsilon^2v''&=&v^3-v+  u^2 v+ \varepsilon^2 u^2 v-\omega u,
	\end{array}
	\right.
	\end{equation}
where the derivative is taken with respect to the $y$ variable. The limit   (\ref{eqConnect})  remains the same.

Motivated by \cite{barankov}, since we expect that $u^2+v^2\to 1$  as $\varepsilon\to 0$, we express $(u,v)$ in polar coordinates as
\begin{equation}\label{eqPolar}
u=R\cos \frac{\varphi}{2},\ \ v=R\sin \frac{\varphi}{2},
\end{equation}
for $R>0$ and $0<\varphi<\pi$. Then,  system (\ref{eql1l2}) for the equilibria decouples into
\begin{equation}\label{eqEquilNew}
  \bar R=1\ \ \textrm{and}\ \ \sin \bar \varphi=\frac{2\omega}{\Lambda-1}.
\end{equation} Under (\ref{eqIneq1}), let $\bar{\varphi}\in (0,\frac{\pi}{2})$ denote the unique solution of (\ref{eqEquilNew}).
 We write (\ref{eqEqGenEpsilon})-(\ref{eqConnect}) equivalently as
\[
\begin{split}
\varepsilon^2 \left[R''- \frac{R}{4} (\varphi')^2\right]=&R^3-R+\frac{\varepsilon^2}{2}R^3\sin^2\varphi-\omega R \sin \varphi,
\end{split}
\]
\[
\begin{split}
\frac{\varepsilon^2}{2} \left(R \varphi''+2  R'\varphi'\right)=&\frac{\varepsilon^2}{4} R^3\sin(2\varphi)-\omega R \cos \varphi;
\end{split}
\]
\[
R\to 1 \ \ \textrm{as}\ y\to \pm \infty,
\]
\[
\varphi \to \pi-\bar{\varphi}\ \textrm{as}\ y\to -\infty,\ \ \varphi \to \bar{\varphi}\ \textrm{as}\ y\to +\infty.
\]
Subsequently, we blow-up the neighborhood near $R=1$ by setting
\begin{equation}\label{eqR}R=1-\varepsilon^2 w, \end{equation}
and get the equivalent problem:
\[
\begin{split}
-\varepsilon^2w''- \frac{1}{4}(1-\varepsilon^2 w) (\varphi')^2 =&(1-\varepsilon^2 w)(\varepsilon^2w^2-2w)\\&+\frac{1}{2} (1-\varepsilon^2 w)^3\sin^2\varphi-\frac{\omega}{\varepsilon^2}(1-\varepsilon^2w)\sin \varphi,
\end{split}
\]
\[
\begin{split}
(1-\varepsilon^2 w) \varphi''-2 \varepsilon^2 w'\varphi'=&\frac{1}{2}(1-\varepsilon^2 w)^3\sin (2\varphi)-2\frac{\omega}{\varepsilon^2}(1-\varepsilon^2w)\cos \varphi;
\end{split}
\]
\[
w\to 0 \ \ \textrm{as}\ y\to \pm \infty,
\ \
\varphi \to \pi-\bar{\varphi}\ \textrm{as}\ y\to -\infty,\ \ \varphi \to \bar{\varphi}\ \textrm{as}\ y\to +\infty.
\]
Now we can define
\begin{equation}\label{eqTransFi}
w_1=w,\ w_2=\varepsilon w_1',\ \varphi_1= \varphi,\ \varphi_2=\varphi_1',
\end{equation}
and get the following equivalent slow system, with $(w_1,w_2)$ being the fast variables and $(\varphi_1,\varphi_2)$ the slow ones:
\begin{equation}\label{eqSF1}\left\{\begin{array}{rcl}
    \varepsilon w_1' & = & w_2, \\
      &   &   \\
    \varepsilon w_2' & = & - \frac{1}{4}(1-\varepsilon^2 w_1) \varphi_2^2 -(1-\varepsilon^2 w_1)(\varepsilon^2w_1^2-2w_1) \\
      &   & - \frac{1}{2}(1-\varepsilon^2 w_1)^3\sin^2\varphi_1+\frac{\omega}{\varepsilon^2}(1-\varepsilon^2w_1)\sin \varphi_1, \\
      &   &   \\
    \varphi_1' & = & \varphi_2, \\
      &   &   \\
    \varphi_2' & = & \frac{2 \varepsilon w_2\varphi_2}{1-\varepsilon^2 w_1}+\frac{1}{2}(1-\varepsilon^2 w_1)^2\sin (2\varphi_1)-2\frac{\omega}{\varepsilon^2}\cos \varphi_1,
  \end{array}\right.
\end{equation}together with the conditions
\begin{equation}\label{eqSFbdryu1}
\left\{\begin{array}{l}
  w_1,\ w_2\to 0 \ \ \textrm{as}\ y\to \pm \infty, \\
    \\
 \varphi_1 \to \pi-\bar{\varphi}\ \textrm{as}\ y\to -\infty,\ \ \varphi_1 \to \bar{\varphi}\ \textrm{as}\ y\to +\infty,\ \ \varphi_2\to 0 \ \textrm{as}\ y\to \pm \infty.
\end{array}\right.
\end{equation}

%\subsection{Analysis at the equilibria}\label{subsubLinearization}
It is easy to check that the eigenvalues of the linearization of (\ref{eqSF1}) at both   equilibria $(0,0,\pi-\bar{\varphi},0)$ and $(0,0,\bar{\varphi},0)$ that we wish to connect are
\begin{equation}\label{eqEVs1}
\pm \frac{\sqrt{2}}{\varepsilon}+O(1), \  \pm \sqrt{1-4c_0^2}+O(\varepsilon) \ \textrm{as}\ \varepsilon\to 0.
\end{equation}Therefore, each of these equilibria is a saddle with two-dimensional (global) stable and unstable manifolds. In light of (\ref{eqConnectNewest}), we will be interested in the unstable manifold $W^u(0,0,\pi-\bar{\varphi},0)$ of $(0,0,\pi-\bar{\varphi},0)$ and the stable manifold $W^s(0,0,\bar{\varphi},0)$ of $(0,0,\bar{\varphi},0)$.
% Moreover, as associated eigenfunctions we can choose the following:
%\begin{equation}\label{eqEigen}
%\left(\pm \frac{1}{\sqrt{2}},1,0,0 \right),\ \left(0,0,\pm \lambda,1 \right)\ \ \textrm{and}\ \ \left(\pm \frac{\lambda}{\sqrt{2}},1,0,0 \right),\ \left(0,0,\pm 1,1 \right),
%\end{equation}
%respectively.

By virtue of (\ref{eqHamil}), we find that the Hamiltonian $\hat{H}$ defined by
 \begin{equation}\label{eqRHS}\begin{split}\frac{2}{\varepsilon^2}\hat{H}
 =& \varepsilon^2 w_2^2 +\frac{1}{4}(1-\varepsilon^2 w_1)^2\varphi_2^2  \\
 & -\frac{\varepsilon^2}{2}(2w_1-\varepsilon^2 w_1^2)^2- \left[\frac{(1-\varepsilon^2 w_1)^2}{2} \sin \varphi_1-c_0-\tilde{\omega}(\varepsilon)\right]^2
\end{split}
\end{equation}
is conserved along solutions of (\ref{eqSF1}). In particular, $\hat{H}=0$ holds along solutions that satisfy
one of the asymptotic behaviours in (\ref{eqSFbdryu1}) at minus or plus infinity. In other words,
\begin{equation}\label{eqH=01}
\hat{H}=0\ \ \textrm{on}\ \ W^u(0,0,\pi-\bar{\varphi},0)\cup W^s(0,0,\bar{\varphi},0).
\end{equation}
\subsection{The slow (critical) manifold $\mathcal{M}_0$ and the reduced system}\label{secSlowRed1}
Formally setting $\varepsilon=0$ in (\ref{eqSF1}) yields the  \emph{slow limit system}:
\begin{equation}\label{eqSFlimit1}\left\{\begin{array}{rcl}
    0 & = & w_2, \\
      &   &   \\
    0 & = & - \frac{1}{4} \varphi_2^2

    +2w_1  - \frac{1}{2}\sin^2\varphi_1+ c_0\sin \varphi_1, \\
      &   &   \\
    \varphi_1' & = & \varphi_2, \\
      &   &   \\
    \varphi_2' & = &  \frac{1}{2} \sin (2\varphi_1)-2c_0\cos \varphi_1.
  \end{array}\right.
\end{equation}
By solving the  first two equations for $w_1,\ w_2$ we can determine the \emph{slow manifold}:
\begin{equation}\label{eqCritiqMan}
\mathcal{M}_0=\left\{w_1=\frac{1}{8}{\varphi_2^2 +   \frac{1}{4}\sin^2\varphi_1-\frac{c_0}{2}\sin\varphi_1},\ w_2=0,\ (\varphi_1,\varphi_2)\in \mathbb{R}^2 \right\}.
\end{equation}

The last two equations of (\ref{eqSFlimit}) compose the   \emph{reduced system} which defines a flow on the critical manifold $\mathcal{M}_0$.
Coupled with the $\varepsilon=0$ limit of the asymptotic behaviour (\ref{eqSFbdryu1}), this gives rise to  the   \emph{reduced heteroclinic connection problem}:
\begin{equation}\label{eqRB}
\left\{\begin{array}{l}\varphi_1'=\varphi_2,\\
         \varphi_2'  = \cos\varphi_1(\sin \varphi_1-2c_0); \\
           \\
          (\varphi_1,\varphi_2) \to ( \pi-\bar{\varphi}_0,0)\ \textrm{as}\ y\to -\infty,\ \ (\varphi_1,\varphi_2) \to (\bar{\varphi}_0,0)\ \textrm{as}\ y\to +\infty,
       \end{array}\right.
\end{equation}
where
\begin{equation}\label{eqPhi0}
  \sin \bar{\varphi}_0=2c_0,\ \ \bar{\varphi}_0\in (0,\frac{\pi}{2}),
\end{equation}
because of (\ref{eqEquilNew}).

\subsection{The singular heteroclinic connection}\label{subsecUo1}
The  reduced system is a conservative hamiltonian  system. More precisely, the hamiltonian
\begin{equation}\label{eqHamilRedu}
  H_{red}=\frac{1}{2}\varphi_2^2-\frac{1}{2}(\sin \varphi_1-2c_0)^2
\end{equation}
is constant along its solutions. Based on this, we can show the following.\begin{pro}\label{proReduced2}
There exists a unique solution $(\varphi_{1,0},\varphi_{2,0})$ of (\ref{eqRB}) such that \begin{equation}\label{eqPhi101}\varphi_{1,0}(0)=\frac{\pi}{2}.\end{equation}
Moreover, we have \begin{equation}\label{eqPhi102}
                     \varphi_{2,0}=\varphi_{1,0}'<0,
                   \end{equation} and \begin{equation}\label{eqPhi103}
                                                           \varphi_{1,0}(y)+\varphi_{1,0}(-y)\equiv \pi.
                                                         \end{equation}\end{pro}
                                                         \begin{proof}Since $H_{red}$ above is equal to 0 along solutions that satisfy the asymptotic behaviour in (\ref{eqRB}),
                                                         the desired solution satisfies
                                                         \[
                                                         \varphi'=2c_0-\sin \varphi,\  \ \varphi(0)=\frac{\pi}{2}.
                                                         \]Then, the proof of Proposition \ref{proU0} carries over straightforwardly.
                                                         \end{proof}

                                                         We note that $(\pi-\bar{\varphi}_0,0)$ and $(\bar{\varphi}_0,0)$ are saddle equilibria for (\ref{eqRB}) with the same eigenvalues $\pm \sqrt{1-4c_0^2}$.
The trajectory $(\varphi_{1,0},\varphi_{2,0})$ lies in the intersection of the unstable manifold $W^u(\pi-\bar{\varphi}_0,0)$ of $(\pi-\bar{\varphi}_0,0)$ and the stable manifold $W^s(\bar{\varphi}_0,0)$ of $(\bar{\varphi}_0,0)$. We will only be concerned with these parts of the aforementioned invariant manifolds.
The lifting of $(\varphi_{1,0},\varphi_{2,0})$ on $\mathcal{M}_0$ furnishes a \emph{singular heteroclinic connection}.% and we denote it by $(p_0,q_0,u_0,z_0)$.

\subsection{Normal hyperbolicity of the slow manifold}As in Subsection \ref{subsecNH}, to check that the slow manifold $\mathcal{M}_0$ is \emph{normally hyperbolic}
we have to examine the linearization of the righthand side of the first two equations in (\ref{eqSFlimit}) with respect to $(w_1,w_2)$, at any point $(w_1,w_2,\varphi_1,\varphi_2)$ on $\mathcal{M}_0$. The aforementioned linearization is
\[
\left(\begin{array}{cc}
        0 & 1 \\
        2 & 0
      \end{array}
 \right),
\]whose eigenvalues   $ \pm\sqrt{2}$ are not on the imaginary axis. Hence, the slow manifold $\mathcal{M}_0$ is indeed normally hyperbolic.

\subsection{Local persistence of $\mathcal{M}_0$: The invariant manifold $\mathcal{M}_\varepsilon$}\label{secMeps1} Let $\mathcal{K}$ be a compact, simply connected domain in the $(\varphi_1,\varphi_2)$ plane which contains the heteroclinic orbit $(\varphi_{1,0},\varphi_{2,0})$, and whose boundary is a
$C^\infty$ curve. As in Subsection \ref{secMeps}, by Fenichel's first theorem, we deduce that
the restriction of $\mathcal{M}_0$ over $\mathcal{K}$ perturbs smoothly for small $\varepsilon \geq 0$ to a locally invariant, normally hyperbolic manifold $\mathcal{M}_\varepsilon$ for (\ref{eqSF1}).
More precisely,  given an integer $m\geq 1$,  there is an $\varepsilon_0>0$ and functions $h_i(\varphi_1,\varphi_2,\varepsilon)\in C^m\left(\mathcal{K}\times [0,\varepsilon_0) \right)$,
$i=1,2$,  such that the manifold $\mathcal{M}_\varepsilon$ described by
\begin{equation}\label{eqh1}
w_1=\frac{1}{8}{\varphi_2^2 +   \frac{1}{4}\sin^2\varphi_1-\frac{c_0}{2}\sin\varphi_1}+\varepsilon h_1(\varphi_1,\varphi_2,\varepsilon),
\ \ w_2=\varepsilon h_2(\varphi_1,\varphi_2,\varepsilon),\ \ (\varphi_1,\varphi_2)\in\mathcal{K},\end{equation}
is a normally hyperbolic, locally invariant manifold for (\ref{eqSF1}) if $\varepsilon\in (0,\varepsilon_0)$.

 By the normal hyperbolicity of $M_\varepsilon$, and by possibly decreasing the value of $\varepsilon_0>0$, we deduce that the equilibria $(0,0,\pi-\bar{\varphi},0)$ and $(0,0,\bar{\varphi},0)$ of (\ref{eqSF1})  lie on $\mathcal{M}_\varepsilon$, i.e.,
\begin{equation}\label{eqEquiv01}\begin{array}{c}
                                  h_1\left(\pi-\bar{\varphi},0,\varepsilon\right)= h_1\left(\bar{\varphi},0,\varepsilon\right)=-(c_0+\tilde{\omega}(\varepsilon))\frac{\tilde{\omega}(\varepsilon)}{\varepsilon}, \\ \\
                                  h_2\left(\pi-\bar{\varphi},0,\varepsilon\right)= h_2\left(\bar{\varphi},0,\varepsilon\right)=0,
                                \end{array}
 \end{equation}for $\varepsilon \in (0,\varepsilon_0)$. %This is because every invariant set of (\ref{eqSF}) in a sufficiently small $\varepsilon$-independent neighborhood of $\mathcal{M}_0$ must be on $\mathcal{M}_\varepsilon$.
%(having identified $\bar{\varphi}$ with $\bar{\varphi}_0$ if $\varepsilon=0$).

\subsection{Proof of Theorem \ref{thm2}}
The above leads us to the main result of this section, from which Theorem \ref{thm2} follows :
\begin{thm}\label{thm11}
For each $\varepsilon>0$ sufficiently small, there is a heteroclinic solution\\ $(w_{1,\varepsilon}$,$w_{2,\varepsilon}$,$\varphi_{1,\varepsilon}$,$\varphi_{2,\varepsilon})$ of (\ref{eqSF1}) satisfying (\ref{eqSFbdryu1}) which lies on $\mathcal{M}_\varepsilon$ such that
\begin{equation}\label{eqEstims11}
     \varphi_{i,\varepsilon}=\varphi_{i,0}+O(\varepsilon),\ \  i=1,2,
   \end{equation}and
   \begin{equation}\label{eqproduct1}
     |\varphi_{1,\varepsilon}-\pi+\bar{\varphi}|\cdot|\varphi_{1,\varepsilon}-\bar{\varphi}|+|\varphi_{2,\varepsilon}|=O(1)e^{-\left( \sqrt{1-4c_0^2}+O(\varepsilon)\right)|x|},
   \end{equation}
uniformly in $\mathbb{R}$, as $\varepsilon\to 0$, where $(\varphi_{1,0},\varphi_{2,0})$ is as in Proposition \ref{proReduced2}. Furthermore,
\begin{equation}\label{eqz>01}
  \varphi_{2,\varepsilon}<0.
\end{equation}
More precisely, the following estimates hold:
\begin{equation}\label{eqexp221}
   \begin{array}{c}
     w_{1,\varepsilon}=\frac{1}{8}{\varphi_{2,\varepsilon}^2 +   \frac{1}{4}\sin^2\varphi_{1,\varepsilon}-\frac{c_0}{2}\sin\varphi_{1,\varepsilon}}+O(\varepsilon)e^{-\left(\sqrt{1-4c_0^2}+O(\varepsilon)\right)|x|}, \\ \\

     w_{2,\varepsilon}=O(\varepsilon)e^{-\left(\sqrt{1-4c_0^2}+O(\varepsilon)\right)|x|},
   \end{array}
\end{equation}
uniformly in $\mathbb{R}$, as $\varepsilon\to 0$.
\end{thm}
\begin{proof}
  The proof proceeds along the lines of that of Theorem \ref{thm1}, so we will just provide a sketch.

  Let $(w_{1,\varepsilon}^-,w_{2,\varepsilon}^-,\frac{\pi}{2},\varphi_{2,\varepsilon}^-)\in W^u(0,0,\pi-\bar{\varphi},0)$ and
  $(w_{1,\varepsilon}^+,w_{2,\varepsilon}^+,\frac{\pi}{2},\varphi_{2,\varepsilon}^+)\in W^s(0,0, \bar{\varphi},0)$ be the two points that we wish to show that coincide, provided that $\varepsilon>0$ is sufficiently small. In the limit $\varepsilon\to 0$ these collapse to some point $\left(w_1^0,w_2^0,\frac{\pi}{2},\varphi_{2,0}(0)\right)$ on $\mathcal{M}_0$ with $\left(\frac{\pi}{2},\varphi_{2,0}(0) \right)$ on $W^u(\pi-\bar{\varphi}_0,0)\cap W^s( \bar{\varphi}_0,0)$, as defined in Subsection \ref{subsecUo1} (also recall (\ref{eqPhi101})).

The corresponding mapping to that in (\ref{eqMapping}) is now
\[
G\left(\begin{array}{c}
         w_1 \\
         w_2 \\
         \varphi_1 \\
         \varphi_2 \\
         \varepsilon
       \end{array}
 \right)=\left(\begin{array}{c}
                 w_1-\frac{1}{8}{\varphi_2^2 -   \frac{1}{4}\sin^2\varphi_1+c_0\sin\varphi_1}-\varepsilon h_1(\varphi_1,\varphi_2,\varepsilon) \\
                 w_2-\varepsilon h_2(\varphi_1,\varphi_2,\varepsilon) \\
                 \hat{H}(w_1,w_2,\varphi_1,\varphi_2,\varepsilon) ,
               \end{array}
  \right)
\]
where $h_i$ and $\hat{H}$ were defined in  (\ref{eqEquiv01}) and (\ref{eqRHS}) respectively. By virtue of (\ref{eqH=01}),
\[
G\left(w_{1,\varepsilon}^\pm,w_{2,\varepsilon}^\pm,\frac{\pi}{2},\varphi_{2,\varepsilon}^\pm,\varepsilon\right)=0\ \ \textrm{for small}\ \varepsilon>0.
\]
In fact, the above relation continues to hold for $\varepsilon$ as
\[
G\left(w_1^0,w_2 ^0,\frac{\pi}{2},\varphi_{2,0}(0),0\right)=0.
\]
As in Theorem \ref{thm1}, thanks to (\ref{eqPhi102}) at $y=0$, we can apply the implicit function theorem and conclude.
  \end{proof}
  \subsection*{Acknowledgements}
The second author would like to acknowledge support from the program PSL-Maths and thank the CAMS for its hospitality.
\bibliographystyle{plain}
\bibliography{biblioaacs}
\end{document}